\documentclass[a4paper, 12pt]{article}
\usepackage{amsmath}
\usepackage{amssymb}
\usepackage{enumerate,url}
\usepackage{amsthm}
\usepackage[
      colorlinks=true,    
      urlcolor=blue,    
      menucolor=blue,    
      linkcolor=blue,    
      pagecolor=blue,    
      bookmarks=true,    
      bookmarksopen=true,    
      hyperfootnotes=false,    
      pdfpagemode=UseOutlines    
]{hyperref}

\newtheorem{definition}{Definition}
\newtheorem{theorem}{Theorem}
\newtheorem{lemma}{Lemma}
\newtheorem{proposition}{Proposition}
\newtheorem{corollary}{Corollary}

\newcommand{\R}{\mathbb{R}}
\newcommand{\Z}{\mathbb{Z}}
\newcommand{\Q}{\mathbb{Q}}
\renewcommand{\O}{\mathcal{O}}

\newcommand{\ol}{\overline}

\begin{document}

\title{Sums of units in function fields}

\author{Christopher Frei}
\date{May 17, 2010}
\maketitle

\begin{abstract}
Let $R$ be the ring of $S$-integers of an algebraic function field (in one variable) over a perfect field, where $S$ is finite and not empty. It is shown that for every positive integer $N$ there exist elements of $R$ that can not be written as a sum of at most $N$ units.

Moreover, all quadratic global function fields whose rings of integers are generated by their units are determined.
\end{abstract}

\renewcommand{\thefootnote}{}
\footnote{The original publication is available at \href{http://www.springerlink.com/content/p432163104464636/?p=ec2926e4827144f5a86cc5b400b8d79d&pi=0}{\texttt{www.springerlink.com}}.}
\footnote{DOI: 10.1007/s00605-010-0219-7}
\renewcommand{\thefootnote}{\arabic{footnote}}
\setcounter{footnote}{0}

\section{Introduction}
The connection between the additive structure and the units of certain rings has achieved some attention in the last years. First investigations in this direction were made by Zelinsky \cite{Zelinsky1954}, who showed that, except for one special case, every linear transformation of a vector space is a sum of two automorphisms, and Jacobson \cite{Jacobson1964}, who showed that in the rings of integers of the number fields $\Q(\sqrt{2})$ and $\Q(\sqrt{5})$ every element can be written as a sum of distinct units. Jacobson's work was extended by {\'S}liwa \cite{Sliwa1974}, who proved that there are no other quadratic number fields with this property, and Belcher \cite{Belcher1974}, \cite{Belcher1975}, who investigated cubic and quartic number fields. 

Goldsmith, Pabst and Scott \cite{Goldsmith1998}, investigated similar questions, but without the requirement that the units be distinct. The following definition from \cite{Goldsmith1998} describes quite precisely how the units of a ring $R$ additively generate $R$.

\begin{definition}
Let $R$ be a ring with identity and $k$ a positive integer. An element $r \in R$ is called $k$-good if there are units $e_1$, $\ldots$, $e_k$ of $R$, such that $r = e_1 + \cdots + e_k$. If every element of $R$ is $k$-good then we call the ring $k$-good as well.

The unit sum number $u(R)$ of $R$ is defined as $\min\{k \mid R$ is $k$-good $\}$, if this minimum exists. If the minimum does not exist, but $R$ is additively generated by its units, then we define $u(R) := \omega$. If the units do not generate $R$ additively then we set $u(R) := \infty$.
\end{definition}

By convention, we put $n < \omega < \infty$, for every integer $n$. The case where $R$ is the ring of integers of an algebraic number field has recently been of particular interest. The fact that no ring of integers of an algebraic number field can have a finite unit sum number was proved by Ashrafi and V{\'a}mos \cite{Ashrafi2005} in some special cases, and  by Jarden and Narkiewicz \cite{Jarden2007} in the general case. It is also a consequence of a result obtained independently by Hajdu \cite{Hajdu2007}. Our first theorem is an analogous result for rings of $S$-integers of algebraic function fields over perfect fields. 

Regarding function fields, we use the notation from \cite{Rosen2002} and \cite{Stichtenoth1993}. In particular, an algebraic function field over a field $K$ is a finitely generated extension $F | K$ of transcendence degree $1$. If $K$ is a finite field then $F | K$ is called a global function field. The algebraic closure of $K$ in $F$ is called the (full) field of constants of $F | K$. Following \cite{Stichtenoth1993}, we regard the places of $F | K$ as the maximal ideals of  discrete valuation rings of $F$ containing $K$. In particular, the places $P$ of $F | K$ correspond to (surjective) discrete valuations $v_P : F \to (\Z \cup \{\infty\})$ of $F$ over $K$. Let $n$ be a positive integer. We say that a place $P$ of $F | K$ is a zero of an element $f \in F$ of order $n$, if $v_P(f) = n > 0$, and $P$ is a pole of $f$ of order $n$, if $v_P(f) = -n < 0$. If $S$ is a finite set of places of $F | K$ then the ring $\O_S$ of $S$-integers of $F$ is the set of all elements of $F$ that have no poles outside of $S$. The $S$-units of $F$ are the units of $\O_S$. As a consequence of the definition of $\O_S$, an element $f \in F$ is an $S$-unit if and only if $v_P(f) = 0$ for all places $P$ outside of $S$. The pole [zero] divisor $(f)_\infty$ [$(f)_0$] of an element $f \in F^*$ is the sum of all poles [zeros] of $f$, taken with their respective multiplicities. The height $H(f)$ of $f$ is defined as the degree of its zero divisor, or, equivalently, as the degree of its pole divisor:
\begin{align*}H(f) := \deg (f)_0 &= \sum_P \max\{0, v_P(f) \deg P\}\\ &= - \sum_P \min\{0, v_P(f) \deg P\} = \deg (f)_\infty\text,\end{align*}
where the sums run over all places $P$ of $F | K$.

The following theorem is basically a consequence of Mason's classical work on unit equations in function fields \cite{Mason1986char0}, \cite{Mason1986}.

\begin{theorem}\label{unit_sum_number_infinite}
Let $K$ be a perfect field, $F | K$ an algebraic function field, and $S \neq \emptyset$ a finite set of places of $F|K$. Denote by $\O_S$ the ring of $S$-integers of $F$. Then, for each positive integer $N$, there exists an element of $\O_S$ that can not be written as a sum of at most $N$ units of $\O_S$. In particular, we have $u(\O_S) \geq \omega$.
\end{theorem}

To show that both cases, $\omega$ and $\infty$, occur, we give a complete classification of the unit sum numbers of rings of integers of quadratic function fields over finite fields. The number field analogue of this result was found independently by Belcher \cite{Belcher1974} and Ashrafi and V{\'a}mos \cite{Ashrafi2005}. Results of this kind also exist for cubic and quartic number fields \cite{Filipin2008}, \cite{Tichy2007}, \cite{Ziegler2008}. In the global function field case it turns out that the only quadratic function fields whose rings of integers have unit sum number $\omega$ are real quadratic function fields that are again rational.

\begin{theorem}\label{quadratic_function_fields}
Let $K$ be a finite field, and $F$ a quadratic extension field of the rational function field $K(x)$ over $K$. Denote the integral closure of $K[x]$ in $F$ by $\O_F$. Then the following two statements are equivalent.
\begin{enumerate}[(a)]
\item $u(\O_F) = \omega$.
\item The function field $F | K$ has full constant field $K$ and genus $0$, and the infinite place of $K(x)$ splits into two places of $F | K$.
\end{enumerate}
\end{theorem}

Of course, one can use Theorem \ref{quadratic_function_fields} to obtain explicit criteria similar to those in \cite{Ashrafi2005}, \cite{Belcher1974}, \cite{Filipin2008}, \cite{Tichy2007}, \cite{Ziegler2008}. If $K$ is of odd characteristic then every quadratic extension field of the rational function field $K(x)$ with full constant field $K$ is of the form $F = K(x, y)$, where $y$ satisfies an equation $y^2 = f(x)$, for some separable polynomial $f \in K[X] \setminus K$. It is well known that $F$ is of genus $0$ if and only if $\deg f \in \{1, 2\}$, and that the infinite place of $K(x)$ splits in $F | K$ if and only if $\deg f$ is even and the leading coefficient of $f$ is a square in $K$. We therefore get the following corollary.

\begin{corollary}\label{cor_char_odd}
Let $K$ be a finite field of odd characteristic, and $F = K(x, y)$, where $K(x)$ is a rational function field over $K$ and $y^2 = f(x)$, for some separable polynomial $f \in K[x] \setminus K$. Denote the integral closure of $K[x]$ in $F$ by $\O_F$. Then the following two statements are equivalent.
\begin{enumerate}[(a)]
\item $u(\O_F) = \omega$.
\item The degree of $f$ is $2$ and the leading coefficient of $f$ is a square in $K$.
\end{enumerate}
\end{corollary}

If $K$ is of characteristic $2$ then every separable quadratic extension field of the rational function field $K(x)$ with full constant field $K$ can be written as $F = K(x, y)$, where $y$ satisfies a quadratic equation in Hasse normal form \cite[p. 38]{Hasse1935}. That is,
\begin{equation}\label{hasse_norm_form}y^2 + y = \frac{g(x)}{p_1(x)^{2 n_1 - 1} \cdots p_m(x)^{2 n_m - 1}}\text,\end{equation}
where, $p_1$, $\ldots$, $p_m \in K[X]$ are monic irreducible polynomials and distinct from each other, $n_1$, $\ldots$, $n_m$ are positive integers, $g \in K[X]$ is not divisible by any of the $p_i$, and the infinite place of $K(x)$ is either no pole or a pole of odd order of the right-hand side of \eqref{hasse_norm_form}. (That is, the difference of the degrees of denominator and numerator is non-negative or odd.) We put $B := p_1^{n_1}\cdots p_m^{n_m}$, and $C := g p_1 \cdots p_m$ (cf. \cite{LeBrigand2005}). Then \eqref{hasse_norm_form} becomes
\begin{equation}\label{mod_hasse_norm_form}y^2 + y = \frac{C(x)}{B(x)^2}\text.\end{equation}
Note that $K$ is the full constant field of $F | K$ if and only if $C$ is not constant. Using well-known properties of Artin-Schreier extensions of function fields (for example Proposition III.7.8. from \cite{Stichtenoth1993}), we see that the function field $F | K$ is of genus $0$ if and only if
\begin{equation}\label{a_s_degrees_case_1}\deg B = 0\text\quad\text{ and }\quad \deg C = 1\end{equation}
or
\begin{equation}\label{a_s_degrees_case_2}\deg B = 1\text\quad\text{ and }\quad \deg C \leq 2\text.\end{equation}
In case \eqref{a_s_degrees_case_1} the infinite place of $K(x)$ is ramified in $F | K$, and in case \eqref{a_s_degrees_case_2} the infinite place of $K(x)$ splits in $F | K$ if and only if either $\deg C < 2$, or $\deg C = 2$ and the leading coefficient of $C$ has the form $a^2 + a$, for some $a \in K$. (These are exactly the cases where the projection of $y^2 + y + C(x)/B(x)^2$ to the polynomial ring over the residue class field of the infinite place of $K(x)$ is reducible.) We have shown the following analogue of Corollary \ref{cor_char_odd}.

\begin{corollary}
Let $K$ be a finite field of characteristic $2$, and $F = K(x, y)$, where $y$ satisfies an irreducible quadratic equation \eqref{mod_hasse_norm_form} in Hasse normal form. Denote the integral closure of $K[x]$ in $F$ by $\O_F$. Then the following two statements are equivalent.
\begin{enumerate}[(a)]
\item $u(\O_F) = \omega$.
\item We have $\deg B = 1$ and either $\deg C \leq 1$, or $\deg C = 2$ and the leading coefficient of $C$ is of the form $a^2 + a$, for some $a \in K$. 
\end{enumerate}
\end{corollary}

Note that if $K$ is of characteristic $2$ and $F|K(x)$ is an inseparable quadratic extension then it is purely inseparable, and the infinite place of $K(x)$ is ramified in $F | K$. Therefore, we have $u(\O_F) = \infty$ in this case.

In the number field case, quantitative problems in relation with sums of units have been objects of recent study. The question, how many non-associated algebraic integers with bounded norm in a number field can be written as a sum of exactly $k$ units, has been investigated in \cite{Filipin2008}, \cite{Filipin2008B} and \cite{Fuchs2009}. Similar considerations in the function field case would be of interest.
\section{Proof of Theorem \ref{unit_sum_number_infinite}}
Let $\tilde{K}$ be the full constant field of $F | K$. Since the places of the function fields $F | K$ and $F | \tilde{K}$ are the same, we may assume without loss of generality that $K = \tilde{K}$. We start with the case where $K$ is of characteristic $0$.
\subsection{Characteristic $0$}
The main tool for our proof is a finiteness result on $S$-unit equations by Mason \cite[Lemma 2]{Mason1986char0}.
\begin{lemma}\label{mason_lemma_char_0}
Let $K$ be an algebraically closed field of characteristic $0$, $F | K$ an algebraic function field, and $S$ a finite set of places of $F | K$. Suppose that $u_1$, $\ldots$, $u_k$ are $S$-units in $F$ such that
$u_1 + \cdots + u_k = 1\text,$
and no proper subset of $\{1, u_1, \ldots, u_k\}$ is $K$-linearly dependent. Then we have $H(u_i) \leq A(k)$, for all $1 \leq i \leq k$ and a constant $A(k)$ that depends only on $k$, $S$, and $F | K$.
\end{lemma}
 
Mason even provides an explicit formula for the bound $A(k)$, which was improved by Brownawell and Masser \cite{Brownawell1986}. For our purpose, however, the above lemma is sufficient. Suppose that every element of $\O_S$ is a sum of at most $N$ $S$-units, for some integer $N > 1$. Choose some non-constant $S$-integer $r$ that is not an $S$-unit, and denote the set of zeros of $r$ by $T$. Obviously, $r$ is an $(S \cup T)$-unit, and there is some place $P \in T \smallsetminus S$.

For every positive integer $n$, there exists some $2 \leq k \leq N$, and $S$-units $\varepsilon_1$, $\ldots$, $\varepsilon_k \in \O_S^*$, such that 
$$\varepsilon_1 + \cdots + \varepsilon_k = r^n\text,$$
and no proper subset of $\{\varepsilon_1, \ldots, \varepsilon_k, r^n\}$ is $K$-linearly dependent. Therefore, we have
\begin{equation}\label{unit_equation}\varepsilon_1 / r^n + \cdots + \varepsilon_k / r^n = 1\text,\end{equation}
for $(S \cup T)$-units $\varepsilon_1 / r^n$, $\ldots$, $\varepsilon_k / r^n$, and still no proper subset is $K$-linearly dependent. 

For some algebraically closed field $\Phi \supset F$, let $\ol{K}$ be the algebraic closure of $K$ in $\Phi$. We put $F' := F \ol{K}$ and regard the constant field extension $F' | \ol{K}$ of $F | K$. Let $S'$ be the set of all places of $F' | \ol{K}$ lying over places in $(S \cup T)$.

Then \eqref{unit_equation} is an $S'$-unit equation in $F' | \ol{K}$. Since $F$ and $\ol{K}$ are linearly disjoint over $K$ (see, for example, \cite[Proposition III.6.1]{Stichtenoth1993}), all requirements of Lemma \ref{mason_lemma_char_0} are satisfied. Therefore,
\begin{equation}\label{height_bounded}H(\varepsilon_1 / r^n) \leq A(k) \leq A := \max\{A(2), \ldots, A(N)\}\text.\end{equation}

On the other hand, we have $H(\varepsilon_1 / r^n) \geq | v_{P'}(\varepsilon_1 / r^n) | = n v_{P'}(r) \geq n$, for any place $P'$ of $F' | \ol{K}$ lying over $P$. Here we used that $\varepsilon_1$ is an $S$-unit and $P \notin S$, whence $v_P(\varepsilon_1)=0$. If $n$ is chosen big enough, this contradicts \eqref{height_bounded}.

\subsection{Positive characteristic}
The case of positive characteristic $p$ is similar in spirit, but a little bit more technical. The main problem is that, due to the Frobenius homomorphism, the height of solutions of unit equations is no longer bounded. For if
$$u_1 + \cdots + u_k = 1$$
is a solution of such a unit equation then 
$$u_1^{p^l} + \cdots + u_k^{p^l} = 1$$
as well, for any positive integer $l$. Again, we use a result by Mason. The following lemma is a special form of Lemma 1 from \cite{Mason1986}. 

\begin{lemma}\label{mason_lemma}
Let $K$ be an algebraically closed field of positive characteristic $p$, and $K(z) | K$ a rational function field. Let $F$ be a finite separable extension of $K(z)$, and denote by $\O_F$ the integral closure of $K[z]$ in $F$. 

For each positive integer $k$, there exist bounds $M(k)$, $A(k) \in \R$, depending only on $F | K(z)$ and $k$, such that the following holds: Let $(u_1, \ldots, u_k) \in (\O_F^*)^k$  be a solution of the unit equation
$$u_1 + \cdots + u_k = 0,$$
such that no proper subsum on the left-hand side vanishes. Then there are non-negative integers $m$, $t(1)$, $\ldots$, $t(m)$, a non-zero constant $\eta \in K$, and units $\eta_1$, $\ldots$, $\eta_m \in \O_F^*$, such that
$$u_2 / u_1 = \eta \prod_{j=1}^m \eta_j^{p^{t(j)}}\text.$$
Moreover, we have $m \leq M(k)$, and $H(\eta_j) \leq A(k)$, for all $1 \leq j \leq m$. (As usual, the empty product is interpreted as $1$.)
\end{lemma}

Additionally, we use the following elementary number-theoretical lemma.

\begin{lemma}\label{num_th_lemma}
Let $p$, $M$, $A$ be positive integers. Then there exist infinitely many positive integers $n$ that can not be written in the form
\begin{equation}\label{n_sum_p_powers}n = -\sum_{j=1}^{m} p^{t(j)} k_j\text,\end{equation}
with any integer $0 \leq m \leq M$, integers $k_j$ with $|k_j| \leq A$, and non-negative integers $t(j)$. 
\end{lemma}

\begin{proof}
Let $T$ be any positive integer and $R_T$ the set of residue classes of all positive integers of the form \eqref{n_sum_p_powers} modulo $p^T$. Each residue class in $R_T$ has a representative of the form
$$-\sum_{j=1}^M p^{s(j)} k_j\text,$$
with $k_j$ as in the lemma, and integers $s(j) \in \{0, \ldots, T-1\}$. Obviously, there are at most $(T (2A + 1))^M$ such representatives, which is a polynomial in $T$. On the other hand, there are $p^T$ residue classes modulo $p^T$. If $T$ is chosen big enough then not all residue classes modulo $p^T$ are in $R_T$, and the lemma follows immediately. \qed
\end{proof}

Let $N$ be a positive integer. We construct an element of $\O_S$ that can not be written as a sum of at most $N$ units of $\O_S$. 

Choose some place $P$ of $F|K$ that is not in $S$. The strong approximation theorem permits us to find an $S$-integer $r \in \O_S$ with $v_P(r) = 1$. Let $T$ be the set of zeros of $r$. Then $r$ is obviously an $(S \cup T)$-unit, but no $S$-unit. 

For some algebraically closed field $\Phi \supset F$, let $\ol{K}$ be the algebraic closure of $K$ in $\Phi$, and put $F' := F \ol{K}$. We regard the constant field extension $F' | \ol{K}$ of $F | K$.

Let $S'$ be the set of all places of $F'|\ol{K}$ lying over places in $(S \cup T)$, and choose some places $Q \in S'$ and $R$, $R' \notin S'$ of $F'|\ol{K}$. By the strong approximation theorem, we can find an element $z \in F'$ that satisfies the conditions
\begin{align*}v_R(z) &= 1\text,\\v_{R'}(z) &= |S'|-1\text,\\v_W(z) &= -1\text{, for all } W\in S'\smallsetminus\{Q\}\text{, and}\\v_W(z) &\geq 0\text{, for all places } W\notin S'\cup \{R, R'\}\text{.}
\end{align*}

Since the principal divisor of $z$ must have degree $0$, it follows that $v_Q(z) < 0$. Therefore, the poles of $z$ are exactly the elements of $S'$. Moreover, $z$ is not a $p$-th power, since $p$ does not divide $v_R(z) = 1$. It follows that $F'$ is separable over $\ol{K}(z)$ (use, for example, Proposition III.9.2 (d) from \cite{Stichtenoth1993}) and the integral closure of $\ol{K}[z]$ in $F'$ is exactly $\O_{S'}$, the ring of $S'$-integers of $F'$. 

For any positive integer $k$, let $M(k)$, $A(k)$ be the constants from Lemma \ref{mason_lemma}, for the function field extension $F' | \ol{K}(z)$. In Lemma \ref{num_th_lemma}, put 
$$M := \max\{M(k) \mid 2 \leq k \leq N + 1\}\text,\quad\text{ and }\quad A := \max\{A(k) \mid 2 \leq k \leq N + 1\}\text,$$
and choose some positive integer $n$ that can not be written in the form \eqref{n_sum_p_powers}.

We claim that the element $r^n \in \O_S$ can not be written as a sum of at most $N$ units of $\O_S$. Suppose otherwise; then there is some $2 \leq k \leq N$ and units $\varepsilon_1$, $\ldots$, $\varepsilon_k \in \O_S^*$, such that
$$\varepsilon_1 + \cdots + \varepsilon_k = r^n\text,$$
and no proper subsum on the left-hand side vanishes. Therefore, we get
$$-r^n + \varepsilon_1 + \cdots + \varepsilon_k = 0\text,$$
for $(S \cup T)$-units $-r^n$, $\varepsilon_1$, $\ldots$, $\varepsilon_k$, and still no proper subsum vanishes. Regarded as elements of $F'$, the summands on the left-hand side are $S'$-units. Lemma \ref{mason_lemma} implies that there exist an integer $0 \leq m \leq M$, non-negative integers $t(1)$, $\ldots$, $t(m)$, a constant $\eta \in \ol{K}^*$, and $S'$-units $\eta_1$, $\ldots$, $\eta_m \in \O_{S'}^*$, such that $H(\eta_j) \leq A$, for all $1 \leq j \leq m$, and 
\begin{equation}\label{asdf_1}\varepsilon_1/r^n = -\eta \prod_{j=1}^m \eta_j^{p^{t(j)}}\text.\end{equation}

Let $P' \in S'$ be a place of $F'| \ol{K}$ lying over $P$. Since $K$ is perfect, constant field extensions are unramified, and thus $v_{P'}(r) = 1$. We consider \eqref{asdf_1} in the $P'$-adic valuation:
$$-n = v_{P'}(\varepsilon_1 / r^n) = \sum_{j=1}^m p^{t(j)} v_{P'}(\eta_j)\text.$$
Since $|v_{P'}(\eta_j)|$ are bounded by $H(\eta_j) \leq A$, we found a representation \eqref{n_sum_p_powers}, contrary to our choice of $n$. This completes the proof of Theorem \ref{unit_sum_number_infinite}.

\section{Proof of Theorem \ref{quadratic_function_fields}}\label{proof_quadratic}
\subsection{\emph{(b)} implies \emph{(a)}}
To show that \emph{(b)} implies \emph{(a)}, we prove a more general proposition.

\begin{proposition}\label{s_integers_rational_function_field}
Let $K(x)$ be a rational function field over any perfect field $K$, $n \geq 2$ an integer, and $S = \{P_1, \ldots, P_n\}$ a set of $n$ distinct places of $K(x)$ of degree one. Denote by $\O_S$ the ring of $S$-integers of $K(x)$. Then $u(\O_S) = \omega$.
\end{proposition}

\begin{proof}
By Theorem \ref{unit_sum_number_infinite}, we have $u(\O_S) \geq \omega$, hence it is enough to show that every element of $\O_S$ is a sum of $S$-units. This is clear for $0 \in \O_S$. Let $f \in \O_S \setminus \{0\}$ be a non-zero element. The pole divisor of $f$ has the form
$$(f)_\infty = v_1 P_1 + \cdots + v_n P_n\text,$$
with non-negative integers $v_1$, $\ldots$, $v_n$. If $H(f) = 0$ then $f$ is a constant and nothing is left to prove. Assume that $H(f) > 0$, that is, at least one of the $v_i$ is positive. We construct an $S$-unit $u \in \O_S^*$, such that either $f = u$ or $H(f - u) < H(f)$. Then the proposition follows by induction.

Without loss of generality, we assume that $v_1 > 0$. By exchanging the generating element $x$, if necessary, we can always assure that $P_1$ is the infinite place of $K(x)$. Let $x - \alpha_i \in K[x]$ be the monic local parameter for $P_i$, for each $2 \leq i \leq n$. Then $f$ is of the form
$$f = g(x)\cdot(x - \alpha_2)^{- v_2} \cdots (x - \alpha_n)^{- v_n}\text,$$
where $g \in K[X] \setminus \{0\}$ is some polynomial. Since $-v_1 = v_{P_1}(f)$ is the difference of the degrees of denominator and numerator of $f$, we have $-v_1 = v_2 + \cdots + v_n - \deg g$. Therefore, $g$ is of degree $v_1 + \cdots + v_n$. Let $\lambda$ be the leading coefficient of $g$, and put $u := \lambda (x - \alpha_2)^{v_1}$. Then $u$ is an $S$-unit, and we get
$$f - u = \frac{g - \lambda (x - \alpha_2)^{v_1 + v_2}(x - \alpha_3)^{v_3} \cdots (x - \alpha_n)^{v_n}}{(x - \alpha_2)^{v_2} \cdots (x - \alpha_n)^{v_n}}\text.$$
The degree of the numerator is smaller than the degree of $g$. Therefore, we have $v_{P_1}(f - u) > - v_1$. Also, $v_{P_i}(f - u) \geq - v_i$, for $2 \leq i \leq n$, and $v_P(f - u) \geq 0$, for all places $P \notin S$. Therefore, we have either $f - u = 0$ or $H(f - u) < H(f)$. This concludes our proof.\qed
\end{proof}

Now assume \emph{(b)} and let $S := \{P_1, P_2\}$ be the set of infinite places of $F | K$. Then both $P_1$ and $P_2$ are of degree one, so $F$ is a rational function field over $K$. The integral closure of $K[x]$ in $F$ is exactly the ring of $S$-integers $\O_S$ of $F$, whence \emph{(a)} follows from Proposition \ref{s_integers_rational_function_field}.

\subsection{\emph{(a)} implies \emph{(b)}}
If $K$ is not the full constant field of $F | K$ then $F | K(x)$ is a constant field extension (since it is of degree $2$). Thus, $F = \tilde{K}(x)$, where $\tilde{K}$ is the full constant field of $F | K$. Since then $\O_F = \tilde{K}[x]$, the units of $\O_F$ are constants, so $u(\O_F) = \infty$. Therefore, $K$ is the full constant field of $F | K$. We treat the cases of even and odd characteristic separately.

\subsubsection{Odd characteristic}
Let $p \geq 3$ be the characteristic of $K$. In this case, we always have $F = K(x, y)$, for some $y$ in $F$, satisfying an equation 
$$y^2 = f(x)\text,$$
with a separable polynomial $f \in K[X] \setminus K$. We have $\O_F = K[x, y]$, since the separability of $f$ implies non-singularity of the affine curve $Y^2 = f(X)$. 

The following two lemmata use the notation of the preceding paragraph. The first one is the function field analogue of Lemma 1 from \cite{Belcher1974} and Theorem 7 from \cite{Ashrafi2005}.

\begin{lemma}\label{generated_by_units}
The ring of integers $\O_F$ is generated by units as a $K[x]$-module if and only if there is some $\mu \in K^*$ such that $f + \mu$ is a square in $K[X]$. 

In this case, the unit group $\O_F^*$ is of rank $1$ and there is a fundamental unit of the form $a(x) + y$, for some $a \in K[X]$.
\end{lemma}

\begin{proof}
Our proof is basically the same as Belcher's proof of the number field case. First, assume that $f + \mu = g^2$, for some $\mu \in K^*$ and $g \in K[X]$. This implies that 
$$(g(x) + y)(g(x) - y) = g(x)^2 - f(x) = \mu \in K^*\text,$$
whence $g(x) + y$, $g(x) - y$ are units in $\O_F$. Therefore, $y = (g(x) + y) - g(x) \cdot 1$ is a $K[x]$-linear combination of units of $\O_F$, and since $\O_F$ is generated by $\{1, y\}$ as a $K[x]$-module, we conclude that $\O_F$ is generated as a $K[x]$-module by its units.

Now assume that $\O_F$ is generated by its units as a $K[x]$-module. By Dirichlet's unit theorem (for a version that holds in global function fields see Proposition 14.2 from \cite{Rosen2002}), the group $\O_F^*/K^*$ is a free abelian group of rank $0$ or $1$. Rank $0$ can not happen, since then the group of units in $\O_F$ would be exactly $K^*$, which generates only $K[x]$ as a $K[x]$-module. Therefore, we have a fundamental unit $a(x) + b(x) y$, with some $a$, $b \in K[X]$, $b$ monic, and every unit of $\O_F$ is of the form
$$\lambda (a(x) + b(x) y)^n\text,$$
with some constant $\lambda \in K^*$ and some integer $n$. Since the norm of $a(x) + b(x)y$ is a unit of $K[x]$, hence an element of $K^*$, we have 
$$(a(x) + b(x) y)^{-1} = \kappa (a(x) - b(x) y)\text,$$
for some $\kappa \in K^*$. Let us write $y$ as a $K[x]$-linear combination of units:
$$y = g_0(x) + \sum_{i = 1}^{k_1} g_i(x) (a(x) + b(x) y)^{n_i} + \sum_{i = 1}^{k_2} h_i(x) (a(x) - b(x) y)^{m_i}\text,$$
where $k_1$, $k_2$ and all $n_i$, $m_i$ are positive integers, and all $g_i$, $h_i \in K[X]$. Expanding the right-hand side yields an equation of the form
$$y = g(x) + h(x) b(x) y\text,$$
with polynomials $g$, $h \in K[X]$. By comparing the coefficient of $y$, we get $b(x) \in K[x]^* = K^*$, whence $b = 1$. Since the norm of $a(x) + y$ is a unit in $K[x]$, we get $a^2 - f = \mu \in K^*$, and $f$ is of the desired form.\qed
\end{proof}

\begin{lemma}\label{degrees_of_units}
Let $a \in K[X]$, such that $a(x) + y$ is a unit of $\O_F$. For any non-negative integer $n$, define polynomials $a_n$, $b_n \in K[X]$ via $a_n(x) + b_n(x) y := (a(x) + y)^n$. Then $\deg f$ is even, and for every positive integer $n$, we have 
\begin{equation}\label{degrees}\deg a_n = n (\deg f)/2\quad\text{ and }\quad\deg b_n = (n - 1)(\deg f)/2\text.\end{equation}
\end{lemma}

\begin{proof}
Induction on $n$ proves that $a_n$, $b_n$ are given by the recursive formulas
\begin{equation}\label{a_n_b_n_formula}a_{n+1} = a a_n + b_n f\quad\text{ and }\quad b_{n+1} = a b_n + a_n\text,\end{equation}
with starting values $a_0 = 1$, $b_0 = 0$.

Since $N(a(x) + y) = a(x)^2 - f$ is a constant, it follows that $f$ is of even degree, and $\deg a = (\deg f)/2$. From $a_n(x)^2 - b_n(x)^2 f = N(a(x) + y)^n \in K^*$ we get, for all positive integers $n$,
\begin{equation}\label{degrees_same}\deg b_n = \deg a_n - (\deg f)/2\text,\end{equation}
and
\begin{equation}\label{leading_coeff}\text{the leading coefficients of $a_n^2$ and $b_n^2 f$ coincide.} \end{equation}

By \eqref{degrees_same} and the recursion formulas \eqref{a_n_b_n_formula},
\begin{equation}\label{degrees_less_equal}\deg a_{n+1} \leq \deg a_n + (\deg f)/2\quad\text{and}\quad \deg b_{n+1} \leq \deg b_n + (\deg f)/2\text, \end{equation}
for all positive integers $n$.

We first prove \eqref{degrees} for the cases where $n$ is a power of $2$. We have already seen that the assertion holds for $n = 2^0 = 1$. Since 
$$(a_{2^{l+1}}(x) + b_{2^{l+1}}(x) y) = (a(x) + y)^{2^{l+1}} = (a_{2^l}(x) + b_{2^l}(x) y)^2\text,$$
we get
$$a_{2^{l+1}} = a_{2^l}^2 + b_{2^l}^2 f\quad\text{ and }\quad b_{2^{l+1}} = 2 a_{2^l} b_{2^l}\text.$$
By \eqref{degrees_same} and by induction, we have $\deg a_{2^l}^2 = \deg b_{2^l}^2 f = 2^{l+1} (\deg f)/2$. Assertion \eqref{leading_coeff} and the fact that $p \neq 2$ imply that $\deg a_{2^{l+1}} = 2^{l+1}(\deg f)/2$. Also,
$$\deg b_{2^{l+1}} = \deg a_{2^l} + \deg b_{2^l} = (2^{l+1}-1)(\deg f)/2\text.$$ 

Now let $n$ be an arbitrary positive integer and find the positive integer $l$, such that $2^{l-1} \leq n < 2^l$. We already know that 
$$\deg a_{2^{l-1}} = 2^{l-1}(\deg f)/2\quad\text{ and } \deg a_{2^l} = 2^l (\deg f)/2\text.$$
By $(n - 2^{l-1})$ applications of \eqref{degrees_less_equal}, we get $\deg a_n \leq n (\deg f)/2$. Suppose that we have $\deg a_n < n (\deg f)/2$. Then $(2^l - n)$ applications of \eqref{degrees_less_equal} lead to $\deg a_{2^l} < 2^l (\deg f)/2$, a contradiction.

This and \eqref{degrees_same} yield the desired result.\qed
\end{proof}

Assume \emph{(a)}. We have already seen that $K$ is the full constant field of $F | K$. Clearly, $F$ is generated by its units as a $K[x]$-module. Now Lemma \ref{generated_by_units} implies that $\deg f$ is even, that the unit group $\O_F^*$ is of rank $1$, and that a fundamental unit is of the form $a(x) + y$, for some $a \in K[X]$.

Let $a_n$, $b_n$ be as in Lemma \ref{degrees_of_units}. Then all units of $\O_F$ are given by
\begin{equation}\label{form_of_units}\lambda (a_n(x) + b_n(x) y)\text,\quad\text{ and }\quad\lambda (a_n(x) - b_n(x) y)\text,\end{equation}
for $\lambda \in K^*$ and non-negative integers $n$. Every element of $K[x]$ is a sum of units and can thus be written as a $K$-linear combination of the $a_n(x)$. Since the degrees of all $a_n$ are different from each other, and all divisible by $(\deg f)/2$, it follows that the degree of every polynomial in $K[X]$ is divisible by $(\deg f)/2$ as well, whence $\deg f = 2$. Therefore, the genus of $F | K$ is $0$. It remains to show that the infinite place of $K(x)$ splits into two places of $F | K$. Let $S$ be the set of places of $F | K$ lying over the infinite place of $K(x)$. Then $\O_F = \O_S$, the ring of $S$-integers of $F$. By Proposition 14.2 from \cite{Rosen2002}, the unit group $\O_S^*$ is of rank $|S|-1$. We already know that the rank of $\O_F^*$ is $1$, hence the infinite place of $K(x)$ splits into two places of $F | K$.

\subsubsection{Characteristic $2$}
The only thing left to consider is the case where $K$ is of characteristic $2$. We have already seen that $K$ is the full constant field of $F | K$. Let $g$ be the genus of $F | K$ and assume that 
\begin{equation}\label{genus}g > 0\end{equation}
or that 
\begin{equation}\label{not_split}\text{the infinite place of $K(x)$ is inert or ramified in $F | K$.}\end{equation}
We need to show that not every element of $\O_F$ is a sum of units. This is true if \eqref{not_split} holds. Indeed, we have already seen that, by Dirichlet's unit theorem, the unit group $\O_F^*$ is of rank $s-1$, where $s$ is the number of places of $F | K$ lying over the infinite place of $K(x)$. If there is only one such place then this rank is $0$, whence the unit group $\O_F^*$ only consists of torsion elements. Then $u(\O_F) = \infty$, since the torsion subgroup of $\O_F^*$ is exactly $K^*$, the group of non-zero constants.

Assume from now on that \eqref{genus} holds and \eqref{not_split} does not hold. Then $F | K(x)$ must be a separable extension, since otherwise it is purely inseparable (as it is of degree $2$) and therefore every place is ramified. Separable quadratic extension fields $F$ of $K(x)$ with full constant field $K$ and of genus $g > 0$, such that the infinite place of $K(x)$ splits in $F | K$, can always be written as $F = K(x, y)$, for some $y \in F$ that satisfies an equation 
\begin{equation}\label{char_2_curve}y^2 + B(x) y + C(x) = 0\text,\end{equation}
with polynomials $B$, $C \in K[X] \setminus \{0\}$ having the following properties: The polynomial $B$ is monic, and every prime factor of $B$ is a simple prime factor of $C$. Moreover, we have $\deg B = g + 1$ and $\deg C < 2 g + 2$. This is a slightly modified version \cite[Theorem 1]{LeBrigand2005} of the Hasse normal form for Artin-Schreier extensions \cite[p. 38]{Hasse1935}.

Let us first show that $\O_F = K[x, y]$. This holds if the affine curve given by \eqref{char_2_curve} is non-singular. Let 
$$B = \prod_{i=1}^r B_i^{n_i}$$ be the prime factor decomposition of $B$, with pairwise distinct monic irreducible polynomials $B_i \in K[X]$. Then the polynomial $C$ is of the form
$$C = D \prod_{i=1}^r B_i\text,$$
with some non-zero polynomial $D \in K[X]$ that is not divisible by any $B_i$. Put $G = Y^2 + B(X) Y + C(X) \in K[X, Y]$. The partial derivatives of $G$ are
$$G_X = B'(X) Y + C'(X)\quad\text{ and }\quad G_Y = B(X)\text.$$
Suppose that there are elements $a$, $b$ in some algebraic extension of $K$, such that $G(a, b) = G_X(a, b) = G_Y(a, b) = 0$. Since $B(a) = 0$, there is some $1 \leq k \leq r$ with $B_k(a) = 0$. Therefore, $C(a) = 0$, and thus $b = 0$. Then $G_X(a, 0) = 0$ implies
$$0 = C'(a) = D'(a) \prod_{i=1}^rB_i(a) + D(a)\sum_{i=1}^r B_i'(a) \prod_{j \neq i}B_j(a) = D(a)B_k'(a) \prod_{j \neq k} B_j(a)\text.$$
However, since $D$, $B_k'$, and all $B_j$, for $j \neq k$, are relatively prime to $B_k$, the above product is not $0$, a contradiction. Therefore, the affine curve given by \eqref{char_2_curve} is non-singular, and $\O_F = K[x, y]$.

Note that the conjugate of $y$ over $K(x)$ is $y + B(x)$. The following lemmata use the notation established in \eqref{char_2_curve}. The first one is the analogous result of Lemma \ref{generated_by_units}.

\begin{lemma}\label{generated_by_units_char_2}
The ring of integers $\O_F$ is generated by units as a $K[x]$-module if and only if there is some $\mu \in K^*$ such that the polynomial $Y^2 + B(x) Y + C(x) + \mu \in K[x][Y]$ has a root in $K[x]$. 

In this case, the unit group $\O_F^*$ is of rank $1$, and there is some polynomial $a \in K[X]$ such that $a(x) + y$ is a fundamental unit.
\end{lemma}

\begin{proof}
The proof is similar to the proof of Lemma \ref{generated_by_units}. Assume first that there is some $\mu \in K^*$ and a root $a(x) \in K[x]$ of $Y^2 + B(x) Y + C(x) + \mu$. Then
$$(a(x) + y)(a(x) + B(x) + y) = a(x)^2 + a(x)B(x) + C(x) = \mu \in K^*\text,$$
whence $a(x) + y$ is a unit of $K[x, y] = \O_F$. Therefore, $y = (a(x) + y) + a(x)\cdot 1$ is a $K[x]$-linear combination of units. Since $\O_F$ is generated by $\{1, y\}$ as a $K[x]$-module, it is generated by its units as a $K[x]$-module.

Now assume that the ring of integers $\O_F$ is generated by its units as a $K[x]$-module. The same argument as in the proof of Lemma \ref{generated_by_units} shows that the unit group $\O_F^*$ is of rank $1$, and that there is a fundamental unit $a(x) + b(x) y$, with polynomials $a$, $b \in K[X]$, $b$ monic. Every element of $\O_F^*$ is of the form
$$\lambda (a(x) + b(x) y)^n\text,$$
with $\lambda \in K^*$ and $n \in \Z$. Since the norm of $a(x) + b(x) y$ is in $K^*$, we have 
$$(a(x) + b(x) y)^{-1} = \kappa (a(x) + b(x) B(x) + b(x) y)\text,$$
for some $\kappa \in K^*$. Let us express $y$ as a $K[x]$-linear combination of units:
$$y = g_0(x) + \sum_{i = 1}^{k_1} g_i(x) (a(x) + b(x) y)^{n_i} + \sum_{i = 1}^{k_2} h_i(x) (a(x) + b(x) B(x) + b(x) y)^{m_i}\text,$$
with positive integers $k_1$, $k_2$, $n_i$, $m_i$, and polynomials $g_i$, $h_i \in K[X]$. By comparing the coefficient of $y$, we get $b(x) \in K[x]^* = K^*$, whence $b = 1$. Since the norm of $a(x) + y$ is some $\mu \in K^*$, we have
$$\mu = (a(x) + y)(a(x) + B(x) + y) = a(x)^2 + a(x) B(x) + C(x)\text,$$
as desired.\qed
\end{proof}

Next, we prove an analogue of Lemma \ref{degrees_of_units}.

\begin{lemma}\label{degrees_of_units_char_2}
Let $a \in K[X]$, such that $a(x) + y$ is a unit of $\O_F$. For any non-negative integer $n$, define polynomials $a_n$, $b_n \in K[x]$ via $a_n(x) + b_n(x) y := (a(x) + y)^n$. Then we have, for every positive integer $n$,
\begin{equation}\label{degrees_char_2}\deg a_n \leq n \deg B\quad\text{ and }\quad\deg b_n = (n - 1)\deg B\text.\end{equation}
\end{lemma}

\begin{proof}
Induction on $n$ proves that $a_n$, $b_n$ are given by the recursive formulas
\begin{equation}\label{a_n_b_n_formula_char_2}a_{n+1} = a a_n + b_n C\quad\text{ and }\quad b_{n+1} = a b_n + a_n + b_n B\text,\end{equation}
with starting values $a_0 = 1$, $b_0 = 0$.

Since $\deg C < 2 g + 2 = 2 \deg B$, and
\begin{equation}\label{norm_constant_char_2}a(x)^2 + B(x) a(x) + C(x) = (a(x) + y)(a(x) + B(x) + y) = N(a(x) + y) \in K^*\text,\end{equation}
we get $\deg a \leq \deg B$.

First consider the case where $\deg a < \deg B$. Then $\deg C = \deg B + \deg a$. We use induction to prove the following (in)equalities for every positive integer $n$:
\begin{equation}\label{degrees_char_2_case_1}\deg a_n < \deg b_n + \deg B\quad\text{ and }\quad \deg b_{n+1} = \deg b_n + \deg B\text.\end{equation}
First one checks \eqref{degrees_char_2_case_1} directly for $n = 1$. Assume that both (in)equalities hold for $n$. Then we have $a_{n+1} = a a_n + b_n C$, and $\deg a  + \deg a_n < \deg a + \deg b_n + \deg B = \deg b_n + \deg C$. Therefore,
\begin{align*}\deg a_{n+1} &= \deg b_n + \deg C = \deg b_n + \deg a + \deg B \\ &= \deg b_{n+1} + \deg a < \deg b_{n+1} + \deg B\text.\end{align*}
Similarly,
$$\deg b_{n+2} = \deg (a b_{n+1} + a_{n+1} + b_{n+1} B) = \deg b_{n+1} + \deg B\text.$$
The desired result \eqref{degrees_char_2} now follows by induction from \eqref{degrees_char_2_case_1}. 

Now assume that $\deg a = \deg B$. From \eqref{norm_constant_char_2}, we have $\deg(a^2 + a B) = \deg C < 2 \deg B$, and thus $\deg(a + B) < \deg B = \deg a$. This time, we prove the following equalities for all positive integers $n$:
\begin{equation}\label{degrees_char_2_case_2}\deg a_n = \deg a_{n-1} + \deg a\quad\text{ and }\quad \deg b_n = \deg a_{n-1} \text.\end{equation}
Again, we check the case $n = 1$ directly. Assume \eqref{degrees_char_2_case_2} holds for $n$. Then 
$$\deg b_{n+1} = \deg ((a + B) b_n + a_n) = \deg a_n\text,$$
since $\deg ((a + B) b_n) < \deg a + \deg b_n = \deg a + \deg a_{n-1} = \deg a_n$. Moreover,
$$\deg a_{n+1} = \deg (a a_n + b_n C) = \deg a_n + \deg a\text,$$
since $\deg b_n + \deg C = \deg a_n - \deg a + \deg C < \deg a_n + \deg a$. 

From the first equality of \eqref{degrees_char_2_case_2}, we deduce inductively that $\deg a_n = n \deg B$, whence the second equality of \eqref{degrees_char_2_case_2} implies $\deg b_n = (n-1)\deg B$. \qed
\end{proof}

Suppose that every element of $\O_F$ is a sum of units. Let $a(x) + y$ be the fundamental unit from Lemma \ref{generated_by_units_char_2}, and $a_n$, $b_n$ the polynomials from Lemma \ref{degrees_of_units_char_2}. Then all units of $\O_F$ are of the form 
$$\lambda (a_n(x) + b_n(x) y)\text,\quad\text{ or }\quad\lambda (a_n(x) + b_n(x)B(x) + b_n(x)y)\text,$$
for constants $\lambda \in K^*$ and non-negative integers $n$. Since the degrees of the $b_n$ are all distinct from each other, the only way to represent elements of $K[x]$ as sums of units is as $K$-linear combinations of the
$$(a_n(x) + b_n(x) y) + (a_n(x) + b_n(x) B(x) + b_n(x) y) = b_n(x) B(x)\text.$$
Since $\deg (b_n B) = n \deg B$, and $\deg B = g + 1 > 1$, there is no way to represent $x$ as such a linear combination, which is a contradiction. This completes our proof.

\paragraph*{Added in proof:}
\ There is a simpler way to prove that \emph{(a)} implies \emph{(b)} in Theorem \ref{quadratic_function_fields}, which the author was not aware of when submitting this article. We sketch the argument here: 

Suppose that $u(\O_F) = \omega$. By Dirichlet's unit theorem, the torsion-free part of the unit group of $\O_F$ is of rank at most $1$. Since $\O_F$ is generated by its units as a ring, the rank is $1$. It follows that 
$$\O_F = K[\varepsilon, \varepsilon^{-1}]\text,$$
for some fundamental unit $\varepsilon \in \O_F$. Since the quotient field of $\O_F$ is $F$, we get $F = K(\varepsilon)$, which shows that $F | K$ is of genus $0$ and has full constant field $K$. The fact that the infinite place of $K(x)$ splits into two places of $F | K$ follows in the same way as in Section \ref{proof_quadratic}.

The proof shown in Section \ref{proof_quadratic}, while being significantly longer and more technical than the above argument, has its own merits, especially Lemmata \ref{generated_by_units} and \ref{generated_by_units_char_2}, which show an additional function field analogy of the unit sum number problem in number fields. 

\subsection*{Acknowledgements}
The author is supported by the Austrian Science Foundation (FWF) project S9611-N23.

\bibliographystyle{plain} 
\bibliography{sums_of_units_in_function_fields} 
\noindent Technische Universit\"at Graz\\
Institut f\"ur Analysis und Computational Number Theory \\
Steyrergasse 30, 8010 Graz, Austria\\
E-mail: frei@math.tugraz.at\\
\url{http://www.math.tugraz.at/~frei}
\end{document}